\documentclass[10pt]{amsart}
\usepackage{amsmath,amssymb,amsthm}
\usepackage[dvipdfmx]{hyperref}
\usepackage[all]{xy}
\newtheorem{theorem}{Theorem}[section]
\newtheorem{lemma}[theorem]{Lemma} 
 
\newtheorem{proposition}[theorem]{Proposition} 
 
\newtheorem{question}{Question} 
\newtheorem{corollary}[theorem]{Corollary}
\theoremstyle{definition}
\newtheorem{definition}[theorem]{Definition}
\newtheorem{example}[theorem]{Example}
\newtheorem{remark}[theorem]{Remark}
\newtheorem*{remark*}{Remark}

\newcommand{\N}{\mathbb{N}}
\newcommand{\Z}{\mathbb{Z}}
\newcommand{\R}{\mathbb{R}}
\newcommand{\C}{\mathbb{C}}

\newcommand{\bt}{\mathbf{t}}
\newcommand{\rank}{\mathrm{rank}\,}

\DeclareMathOperator{\Ann}{Ann}

 \makeatletter
    
    \@addtoreset{equation}{section}
  \makeatother
\title[Generalized torsion orders]{Generalized torsion orders of generalized torsion elements}
\author[T.Ito]{Tetsuya Ito}
\address{Department of Mathematics, Kyoto University, Kyoto 606-8502, JAPAN}
\email{tetitoh@math.kyoto-u.ac.jp}
\subjclass[2020]{}
\keywords{}

\begin{document}

\begin{abstract}
A non-trivial element of a group is a generalized torsion element if some products of its conjugates is the identity. The minimum number of such conjugates is called a generalized torsion order. We provide several restrictions for generalized torsion orders by using $G$-invariant norm and Alexander polynomials.
\end{abstract}

\maketitle

%\tableofcontents

\section{Introduction}

An element $g$ of a group $G$ is a \emph{generalized torsion element} if there exists a positive integer $n$ and $x_1,\ldots,x_n \in G$ such that they satisfy
\begin{equation}
\label{eqn:gt-eqn} g^{x_1}g^{x_2}\cdots g^{x_n} = 1.
\end{equation}
Here we put $g^{x}:=xgx^{-1}$.
The \emph{generalized torsion order} $gord(g)$ (often simply called the \emph{order}) is the minimum $n$ such that $g$ satisfies \eqref{eqn:gt-eqn} for some $x_1,\ldots,x_n \in G$.
% When $g$ is not a generalized torsion element we define $gord(g)=\infty$.

The aim of this paper is to investigate the generalized torsion orders. 
More precisely, for $g \in G$, we study the \emph{generalized torsion equation spectrum} $t(g)$, the set of non-negative integers $n$ such that the equation \eqref{eqn:gt-eqn} has a solution. We investigate restrictions for a non-negative integer $n$ to lie in the set $t(g)$.

In \cite[Theorem 2.4]{IMT1} we observed that the \emph{stable commutator length} $scl:G \rightarrow \R_{\geq 0} \cup \{\infty\}$ gives such an restriction. In Theorem \ref{theorem:G-inv-norm} we generalize the scl restriction for $G$-invariant norms, more general classes of functions that has been studied in various contexts.

Although the scl and $G$-invariant norms are useful to investigate generalized torsion elements, their precise computations are hard in general. To get more practical restrictions, we use the Alexander polynomial.

The Alexander polynomial is a (multivariable) polynomial invariant of a group $G$. More precisely, the Alexander polynomial is defined for a group $G$ with surjection $\phi:G \rightarrow \Z^{s}$. Since such a surjection $\phi$ corresponds to a normal subgroup $N$ of $G$ with quotient $G\slash N = \Z^{s}$, we may regard the Alexander polynomial as an invariant of a pair $(G,N)$. 

As a slight generalization, we define the Alexander polynomial $\Delta_{\mathcal{A}}(t_1,\ldots,t_s)$ for an \emph{Alexander tuple} $\mathcal{A}=(G;(X,N,H))$ which is a group $G$ and its normal subgroups $X \subset N \subset H$ having several properties (see Definition \ref{def:Alex-tuple} for details).

For an irreducible polynomial $h(t_1,\ldots,t_s)$, we define
\[ t(h(t_1,\ldots,t_s)) = \{ f(1,\ldots,1) \: | \: f \in (h(\bt)) \mbox{ is positive} \} \subset \N. \]
Here we say that a polynomial $f$ is \emph{positive} if it is non-zero and all the coefficients are non-negative. The following main theorem states that $t(h)$ for an irreducible factor $h$ of the Alexander polynomial gives a restriction of the generalized torsion equation spectrum $t(g)$.\\

\noindent
\textbf{Theorem \ref{theorem:main}.} \textit{ Let $\mathcal{A}=(G;(X,N,H))$ be an Alexander tuple. For an element $g \in N$, if $g \not \in X$, then there exists an irreducible factor $h(t_1,\ldots,t_s)$ of $\Delta_\mathcal{A}(t_1,\ldots,t_s)$ such that $t(g) \subset t(h(t_1,\ldots,t_s))$.\\
}

As applications, we will discuss generalized torsion elements for knot groups. Our results lead to an interesting connection to homology growth of abelian coverings and the generalized torsion elements.\\

\noindent
\textbf{Corollary \ref{cor:tg-knot}.} 
\textit{
Let $K$ be a knot in $S^{3}$.
Let $G=G(K)=\pi_1(S^{3}\setminus K)$ be the knot group and $\Sigma_{k}(K)$ be the $k$-fold cyclic branched covering of $K$.
Assume that $\Sigma_k(K)$ is a rational homology sphere and $k=p^{e}$ is a power of a prime $p$. For $g \not \in [[G,G],[G,G]]$ and $n \in t(g)$ either
\begin{itemize}
\item[(a)] $n \geq |H_1(\Sigma_k(K);\Z)|^{\frac{1}{k-1}}$, or,
\item[(b)] $p$ divides $n$.
\end{itemize}
holds.\\
}

We also show some restrictions for generalized torsion equation spectrum for some special cases. Among them, for torus knots we get the following.\\

\noindent
\textbf{Corollary \ref{cor:torus-knot}.} 
\textit{
Let $K$ be the $(p^{a},q^{b})$-torus knot where $p<q$ are primes. Then for every $g \in G(K)$, $t(g) \subset p\mathbb{N} \cup \mathbb{N}_{\geq q}$.}\\

For example, this says that $(3,7)$-torus knot group does not have a generalized torsion element of generalized torsion order $2,4,5$. On the other hand, it is easy to see that this group has a generalized torsion element of generalized torsion order $3$ (see Proposition \ref{prop:torus-cable}).

This illustrates the \emph{generalized torsion order spectrum} $gord(G)$, the set of generalized torsion orders of a group $G$ can be complicated. Indeed, we will show the following realization result.\\

\noindent
\textbf{Corollary \ref{cor:realization}.} 
\textit{
For every subset $A \subset \N_{\geq 2}$, there exists a countable, torsion-free group $G$ such that $gord(G)=A$.
}

\section*{Acknowledgement}
The author is partially supported by JSPS KAKENHI Grant Numbers 19K03490, 21H04428, 23K03110.

\section{Basics of generalized torsion order}

We summarize basic facts and definitions on generalized torsion elements and its generalized torsion orders that has been appeared elsewhere.

\subsection{Generalized torsion order and generalized torsion equation spectrum}

%We say that a non-trivial element $g$ of a group $G$ satisfies \emph{the order $n$ generalized torsion equation} in $G$ if there exists $x_1,\ldots,x_n$ such that  
%\[ g^{x_1}g^{x_2}\cdots g^{x_n} = 1 \]
%We will mainly focus the following set.

\begin{definition}
The \emph{generalized torsion equation spectrum} $t(g)$ of an element $g \in G$ is the set 
\[ t(g) =\{n \in \N \: | \: g^{x_1}\cdots g^{x_n}=1 \mbox{ for some } x_1,\ldots,x_n \in G\} \]
\end{definition}

By definition, $t(g)$ is a sub-semigroup of $\N$; $n,m \in t(g)$ implies $n+m \in t(g)$. Using the set $t(g)$, generalized torsion elements and its generalized torsion orders are defined as follows. 

\begin{definition}
An element $g$ is a \emph{generalized torsion element} if $t(g) \neq \emptyset$.
The \emph{generalized torsion order} $gord(g)$ is
\[ gord(g)=\min t(g) \]
When $t(g)=\emptyset$ we define $gord(g)=\infty$.
\end{definition}

A torsion element $g$ is a generalized torsion element.
First of all, we discuss several differences between generalized torsion elements and torsion elements.

For a torsion element $g$ of $G$, clearly 
\[ t(g) \supset ord(g)\N \]
holds. Here $ord(g)$ is the order of $g$. In particular, 
\[ gord(g) \leq ord(g) \]
holds. The next example shows that the difference of $gord(g)$ and $ord(g)$ can be arbitrary large.

\begin{example}
\label{exam:ord-vs-gord}
For $m \in \N_{\geq 2} \cup \{\infty\}$
Let 
\[ G_m=\begin{cases} \langle a,b \: | \: bab^{-1}=a^{-1}, a^{m}=1 \rangle & m \in \Z_{\geq 2} \\
\langle a,b \: | \: bab^{-1}=a^{-1} \rangle & m=\infty \end{cases}\] 
Then $ord(a)=m$ but $gord(a)=2$.
\end{example}

The generalized torsion order is often called the \emph{order}. However, since  $ord(g)\neq gord(g)$ in general as Example \ref{exam:ord-vs-gord} shows, it is useful to distinguish the \emph{order} and the \emph{generalized torsion order} when $G$ has a torsion element.

For a subgroup $H$ of $G$ and the inclusion map $i:H \hookrightarrow G$, $h \in H$ is a torsion element of $H$ if and only if $\iota(h)$ is a torsion element of $G$. Furthermore $ord(h)=ord(\iota(h))$. Example \ref{exam:ord-vs-gord} shows that they are far from true for generalized torsion orders.

We will often write $gord(g)$ as $gord_G(g)$ (and $t(g)$ as $t_G(g)$) to emphasize the group $G$. For example, when $H \subset G$ is a subgroup of $G$ and $h \in H$, $gord_H(h)$ means a generalized torsion order of $h$ in the group $H$, whereas $gord_G(h)$ means a generalized torsion order of $\iota(h)$ in the group $G$, where $\iota$ is the inclusion map.

\subsection{Related properties of groups}

In this section we review several notions that are closely related to generalized torsion elements. See \cite{MR} for details.

\begin{definition}
A total ordering $<$ of a group $G$ is a \emph{bi-ordering} if $g < h$ implies $agb < ahb $ for all $a,b,g,h \in G$. Namely, $<$ is invariant under the left and right multiplication of $G$. A group $G$ is \emph{bi-orderable} if $G$ admits a bi-ordering.
\end{definition}
A generalized torsion element serves as an obstruction for bi-orderability; if $G$ has a bi-ordering $<$ then for every non-trivial element $g \in G$, $1<g$ or $g<1$ holds. When $1<g$ then $1=xx^{-1}<xgx^{-1}$ for all $x \in G$ hence product of conjugates of $x$ is not trivial. The case $g<1$ is similar.

\begin{definition}
A group $G$ is 
\begin{itemize}
\item[--] an $R$-group if $x^{n}=y^{n}$ for $x,y \in G$ and $n>1$ implies $x=y$. 
\item[--] an $R^{*}$-group if $G$ has no generalized torsion element.
\item[--] a $TR^{*}$-group if every generalized torsion element of $G$ is a torsion element.
\end{itemize}
\end{definition}

In literatures, $R^{*}$-group is often defined as a group having the property that
$g^{x_1}g^{x_2}\cdots g^{x_n} = h^{x_1}h^{x_2}\cdots h^{x_n}$ implies $g=h$.
This is equivalent to the non-existence of generalized torsion element, because
\begin{align*}
1&=g^{x_1}g^{x_2}\cdots g^{x_n}(h^{x_1}h^{x_2}\cdots h^{x_n})^{-1}\\
&= g^{x_1}g^{x_2}\cdots g^{x_{n-1}}(gh^{-1})^{x_n}(h^{x_1}\cdots h^{x_{n-1}})^{-1} \\
&= (g^{x_1}g^{x_2}\cdots g^{x_{n-1}})(h^{x_1}\cdots h^{x_{n-1}})^{-1}(gh^{-1})^{y_n} \\
&= \cdots \\
&= (gh^{-1})^{y_1}\cdots (gh^{-1})^{y_n}
\end{align*}
where $y_i$ is a suitable element of $G$. 

To see the relevance of generalized torsion elements and $R$-groups, we review the following simplest construction of a generalized torsion element. For $x,y \in G$, let $[x,y]=xyx^{-1}y^{-1}$ be the commutator of $x$ and $y$.

\begin{lemma}\cite[Proposition 3.3]{NR}
\label{lemma:generalized-torsion-simple}
If $[x^{n},y^{m}]=1$ and $[x,y]\neq 1$ for some $n,m>0$, $[x,y]$ is a generalized torsion element and $nm \in t([x,y])$.
\end{lemma}
\begin{proof}
Viewing the commutator as $[x,y]=x (x^{-1})^{y}$, we see that 
\begin{align*}
[x,y][x,y]^{y}\cdots [x,y]^{y^{m-1}} &= x (x^{-1})^y x^{y}(x^{-1})^{y^2} \cdots x^{y^{m-1}}(x^{-1})^{y^m} \\
&= x (x^{-1})^{y^{m}} =[x,y^{m}]
\end{align*}
By a similar argument, we see that $1=[x^n,y^{m}]$ is a product of $n$ conjugates of $[x,y^{m}]$, hence $1=[x^n,y^{m}]$ can be written as a product of $nm$ conjugates of $[x,y]$.
\end{proof}

In this prospect, we can characterize the $R$-groups as groups having no generalized torsion elements obtained by Lemma \ref{lemma:generalized-torsion-simple}.

\begin{proposition}
A group $G$ is an $R$-group if and only if it is torsion-free and $[x^{n},y^{m}]=1$ for $x,y \in G$ $n,m>1$ implies $[x,y]=1$ for all $x,y \in G$.
\end{proposition}
\begin{proof}
Assume that $G$ is an $R$-group. For $x \in G$, $x^n=1=1^n$ implies $x=1$ hence $G$ is torsion free. For $x,y \in G$ assume that $[x^{n},y^{m}]=1$. Then $x^{n}=y^{-m}x^{n}y^{m}=(y^{-m}xy^{m})^{n}$. Since $G$ is an $R$-group, $x=y^{-m}xy^{m}$ so $[x,y^{m}]=1$. Repeating the same argument, we conclude that $[x,y]=1$.

Conversely, assume that $G$ has these two properties. When $x^{n}=y^{n}$ for $x,y$, then $[x^{n},y^{n}]=1$ so $[x,y]=1$. Therefore $(x^{-1}y)^n=x^{-n}y^{n}=1$. Since $G$ is torsion-free, this implies $x=y$.
\end{proof}

$TR^{*}$-groups have the following simple characterization.
\begin{theorem}\cite{It2}
A group $G$ is a $TR^{*}$-group if and only if it is a T(orision)-by-$R^{*}$ group. Namely, there exists a normal subgroup $N$ which is a torsion group such that $G\slash N$ is an $R^{*}$-group. 
\end{theorem}

This characterization says, in most case, a group having a torsion element has a generalized torsion elements which are not torsion elements because usually the set of torsion elements of a group does not form a (normal) subgroup. Thus generalized torsion elements of a group with torsion would be more complicated.
 
\subsection{Generalized torsion elements of generalized torsion order two\label{section:g-ord-two}}

As for a generalized torsion element of generalized torsion order two, we have the following characterization indicating a close connection to $R$-groups.

\begin{proposition}
For $1\neq g \in G$, the followings are equivalent.
\begin{itemize}
\item[(i)] $g$ is a generalized torsion element of generalized torsion order two.
\item[(ii)] $g$ is conjugate to $g^{-1}$.
\item[(iii)] $g=yx^{-1}$ where $x \in G$ and $y \in \sqrt{x^{2}}$. Here $\sqrt{h} = \{f \in G \: | \: f^{2}=h\}$ is the set of square roots of $h$.
\end{itemize}
\end{proposition}
\begin{proof}
The equivalence of (i) and (ii) is clear. We show the equivalence of (i) and (iii).
Let $g$ be a generalized torsion element of generalized torsion order two so $g(xgx^{-1})=1$ for some $x \in G$. Then $(gx)^2=x^{2}$ so $(gx) \in \sqrt{x^{2}}$ and $g=(gx)x^{-1}$. Conversely, when $1\neq g=yx^{-1}$ for $y \in \sqrt{x^{2}}$,
\[ (yx^{-1})(xyx^{-1})(yx^{-1})(xy^{-1}x^{-1}) = 1\]
\end{proof}

\begin{corollary}\cite{HMT}
A group $G$ has a generalized torsion element of generalized torsion order two if and only if there is $x \in G$ such that $\sqrt{x^{2}} \neq \{x\}$. i.e., there exists $x, y \in G$ such that $x^{2}=y^{2}$ but $x \neq y$.
\end{corollary}

\subsection{Monotonicity}

To investigate the set $t(g)$, the next simple observation is useful.

\begin{lemma}[Monotonicity]
\label{lemma:monotonicity}
Let $f:G \rightarrow H$ be a homomorphism. Then for $g \in G$, $t_G(g) \subset t_H(f(g))$. In particular, $gord_G(g) \geq gord_H(f(g))$.
\end{lemma}
\begin{proof}
If $g^{x_1}\cdots g^{x_n}=1$, then $f(g)^{f(x_1)}\cdots f(g)^{f(x_n)}=1$.
\end{proof}

This leads to the following useful consequences.
\begin{corollary}
\label{cor:monotonicity}
{$ $}
\begin{itemize}
\item[(i)] For a prime $p$ and a homomorphism $f:G \rightarrow \Z_p$, if $f(g) \neq 1$ then $t(g) \subset p\mathbb{N}$.
\item[(ii)] If $f:G\rightarrow H$ is a homomorphism and $g \in G$ is a generalized torsion element, $f(g)$ is a generalized torsion element unless $f(g)\neq 1$.
\item[(iii)] If a subgroup $H$ of $G$ is a retract (i.e. there is a map $p:G\rightarrow H$ such that the restriction $p|_H:H \rightarrow H$ is the identity), then for every $h \in H \subset G$, $t_H(h)=t_{G}(h)$ and $gord_H(h) = gord_G(h)$.
\item[(iv)] For every $\phi \in Aut(G)$, $t(g)=t(\phi(g))$ and  $gord(\phi(g))=gord(g)$.
\end{itemize}
\end{corollary}

Although a generalized torsion equation spectrum and a generalized torsion order are sensitive to taking subgroups, we can say the following.

\begin{lemma}
If $g$ is a generalized torsion element of a group $G$, there exists a finitely generated subgroup $H$ of $G$ that contains $g$ such that $gord_H(g) = gord_G(g)$.
\end{lemma} 
\begin{proof}
Let $n=gord_G(g)$.
Take $H$ as the subgroup generated by $g,x_1,\ldots,x_{n} \in G$ where they satisfy $g^{x_1}g^{x_2}\cdots g^{x_n}=1$. Then by the monotonicity, $n=gord_G(g) \leq gord_H(g) \leq n$ hence $gord_H(g)=gord_G(g)=n$.
\end{proof}

\section{Generalized torsion elements and \texorpdfstring{$G$}{G}-invariant norms}

In this section, we show that $G$-invariant norm can be used to evaluate the generalized torsion orders. This is an extension of our previous observation that the stable commutator length gives an lower bound of the generalized torsion order \cite{IMT1}.

\begin{definition}
Let $N$ be a normal subgroup of a group $G$. We say that a function $\nu: N \rightarrow \R_{\geq 0}$ is 
\begin{itemize}
\item[--] \emph{$G$-invariant} if $\nu(gag^{-1}) = \nu(a)$ for all $a \in N$ and $g \in G$.
\item[--] \emph{symmetric} if $\nu(a^{-1})=\nu(a)$ for all $a$.
\item[--] \emph{homogeneous} if $\nu(a^{n})=n \nu(a)$ for all $a$ and $n \in \N$.
\item[--] a \emph{norm} if $\nu(ab)\leq \nu(a)+\nu(b)$ for all $a,b \in N$.
\item[--] a \emph{quasi-norm} if there exists a constant $D_{\nu} \geq 0$ such that $\nu(ab) \leq \nu(a) + \nu(b) + D_{\nu}$ for all $ab \in N$. We call $D_{\nu}$ the \emph{defect} of $\nu$. 
\end{itemize}
\end{definition}

We often allow $\nu$ to take the value $\infty$.
When $N=G$, a $G$-invariant norm is usually called a \emph{conjugation-invariant norm} of the group $G$. Such a norm has been studied in several places. See \cite{BIP} for the relation to geometry.

For $\nu:N\rightarrow \R$, we define its \emph{symmetrization} $\nu^{s}:N\rightarrow \R$ by $\nu^{s}(a)=\frac{\nu(a)+\nu(a^{-1})}{2}$. Symmetrization preserves the property that $\nu$ is $G$-invariant, homogeneous, norm, quasi-norm. 

\begin{remark}
\begin{itemize}
\item[(i)] Although we call $\nu$ a norm, it is actually a \emph{semi-norm} since we do not require $\nu(g)=0$ iff $g=1$. Indeed, we even do not assume $\nu(1)=0$.
\item[(ii)] If $\nu$ is a quasi-norm, then $\nu+D_{\nu}:G \rightarrow \R_{\geq 0}$ given by $(\nu+D_{\nu})(g)=\nu(g) + D_{\nu}$ is a norm. In particular, if $\nu$ is a $G$-invariant quasi-norm then $\nu + D_{\nu}$ is $G$-invariant norm.
\end{itemize}
\end{remark}

\begin{example}[$G$-invariant quasimorphism]
A map $\phi:N \rightarrow \R$ is a \emph{quasimorphism} if
\[ D_{\phi}=\sup_{a,b \in N}|\phi(ab)-\phi(a)-\phi(b)| <\infty \]
$D_{\phi}$ is called the \emph{defect} of $\phi$. A quasimorphism $\phi$ is $G$-invariant if $\phi(gag^{-1})=\phi(a)$.
The absolute value $|\phi|$ of a $G$-invariant quasimorphism $\phi$ gives a $G$-invariant quasi-norm with defect $D_{\phi}$. Thus $|\phi| + D_{\phi}: N \rightarrow \R$ given by $(|\phi|+D_{\phi})(g)= |\phi(g)| + D_{\phi}$ is a $G$-invariant norm.
\end{example}

\begin{example}[Mixed commutator length]
The \emph{mixed commutator length} $cl_{G,N}(g)$ of an element $g \in [G,N]$ is the minimum number of commutators of the form $[x,a]$ or $[a,x]$ ($x \in G, a \in N$) whose product is equal to $g$. Clearly, the mixed commutator length $cl_{G,N}$ is a $G$-invariant symmetric norm.
\end{example}

For a $G$-invariant quasi-norm $\nu$, its \emph{stabilization} (or, \emph{homogenization}) $\overline{\nu}:G \rightarrow \R_{\geq 0}$ is defined by 
\[ \overline{\nu}(g)=\lim_{n \to \infty}\frac{\nu(g^{n})}{n}. \]
Then $\overline{\nu}$ is homogeneous and $G$-invariant. The next lemma gives a sufficient condition for $\overline{\nu}$ to be a quasi-norm.

\begin{lemma}
\label{lemma:stabilization}
For a $G$-invariant quasi-norm $\nu$ of $N$,
\[ \overline{\nu}(gh) \leq \overline{\nu}(g)+\overline{\nu}(h) +\frac{1}{2}\sup_{s,t \in N}\nu([x,y]) \]
\end{lemma}

To prove the lemma, we use the following.
\begin{lemma}
\label{lemma:comm-power}
For $g_1,\ldots,g_n \in G$ and $k>0$, 
\[ (g_1g_2\cdots g_n)^{2k} = g_1^{2k} g_2^{2k} \cdots g_n^{2k} \cdot ((n-1)k  \mbox{ commutators}) \]
Similarly,
\[ (g_1g_2\cdots g_n)^{2k-1} = g_1^{2k-1} g_2^{2k-1} \cdots g_n^{2k-1} \cdot ((n-1)k  \mbox{ commutators}) \]
\end{lemma}
\begin{proof}
%In the following, by $g^{\ast}$ we mean the word of the form $xgx^{-1}$ for some $x \in G$. Similarly, by $[\ast, g]^{\ast}$ we mean the word of the form $y[x,g]y^{-1}$ for some $x,y \in G$.
We prove the lemma by induction on $n$. the case $n=2$ is well-known; see \cite[Lemma 2.24]{Ca} for example. For $n>1$, by induction 
\begin{align*}
(g_1g_2\cdots g_n)^{2k} &= (g_1 (g_2\cdots g_n))^{2k}\\
&= g_1^{2k}(g_2\cdots g_n)^{2k}\cdot (k  \mbox{ commutators}) \\
&= \cdots\\
&= g_1^{2k}g_2^{2k} \cdots g_n^{2k} \cdot ((n-1)k  \mbox{ commutators})
\end{align*}
The $2k-1$ case is similar. 
\end{proof}

\begin{proof}[Proof of Lemma \ref{lemma:stabilization}]
By Lemma \ref{lemma:comm-power}, $(gh)^{2n} = g^{2n}h^{2n}( n \mbox{ commutators})$ hence
\[ \nu((gh)^{2n}) \leq \nu(g^{2n}) + \nu(h^{2n}) + n\sup_{t,s \in N}\nu([t,s]) \]
\end{proof}

Then we have the following estimate of generalized torsion order.

\begin{theorem}
\label{theorem:G-inv-norm}
Let $N$ be a normal subgroup of a group $G$.
Assume that $g \in N$ and $n \in t_G(g)$. Then for a $G$-invariant norm $\nu$ of $N$, if $\sup\{\nu([x,t]) \: | \: x \in G, t \in N \} <\infty$, then 
\[ \overline{\nu}(g) \leq \frac{n-2}{2n}\sup\{\nu([t,s]) \: | \: t,s \in N \} \]
\end{theorem}
\begin{proof}
Since $g \in N$ satisfies the order $n$ generalized torsion equation, there exists $x_1,\ldots,x_{n-1} \in G$ such that 
\[ g^{-1} = g^{x_1}\cdots g^{x_{n-1}} \]
Therefore by taking $2k$-th powers for $k>0$, by Lemma \ref{lemma:comm-power}
\[ 
g^{-2k} = (g^{x_1})^{2k}(g^{x_2})^{2k}\cdots (g^{x_{n-1}})^{2k} \cdot ( (n-2)k \mbox{ commutators}) \\
\]
so
\begin{align*}
g^{-2nk} = [x_1,g^{2k}]^{\ast}[x_2,g^{2k}]^{\ast}\cdots [x_{n-1},g^{2k}]^{\ast} \cdot ( (n-2)k \mbox{ commutators}) \\
\end{align*}
Here $[x,y]^{*}$ means a conjugate of $[x,y]$. Furthermore, the $(n-2)k$ commutators are actually commutators of elements of $N$ (see the proof of Lemma \ref{lemma:comm-power}).
Therefore
\begin{align*}
\frac{\nu(g^{2nk})}{2nk} & \leq \frac{n-1}{2nk} \sup_{x\in G, t \in N}\nu([x,t]) +\frac{(n-2)}{2n} \sup_{s,t \in N}\nu([s,t])
\end{align*}
By taking $k \to \infty$ we get the desired inequality.
\end{proof}

Theorem \ref{theorem:G-inv-norm} makes sense only if $\sup_{s,t \in N}\nu([s,t]) <\infty$. In this case, the assumption $\sup\{\nu([x,t]) \: | \: x \in G, t \in N \} <\infty$ is automatically satisfied.

\begin{example}[Stable mixed commutator length \cite{KKMMM1,KKMMM2}]
The \emph{stable mixed commutator length} $scl_{G,N}(g)$ is the stabilization of the mixed commutator length $cl_{G,N}$. When $N=G$, the (stable) mixed commutator length is called the \emph{(stable) commutator length} of $G$ denoted by $cl_G$ and $scl_G$, respectively.

For an element $g \in G$ such that $g^{\ell} \in [G,N]$ for some $\ell>0$ one can define the stable mixed commutator length $scl_{G,N}(g)$ by $scl_{G,N}(g) = \frac{{cl}_{G,N}(g^{\ell})}{\ell}$.
\end{example}

Applying Theorem \ref{theorem:G-inv-norm} for mixed commutator length or $G$-invariant homogeneous quasimorphisms we get the following. 

\begin{corollary}
\label{cor:mix-scl}
If $g \in N$ and $n \in t_G(g)$
\begin{equation}
\label{eqn:scl}
scl_{G,N}(g) \leq \frac{n-2}{2n} \ \left( < \frac{1}{2} \right)
\end{equation}
and
\begin{equation}
\label{eqn:q-hom}
|\phi(g)| \leq \frac{n-2}{2n}\left(\sup_{s,t \in N}|\phi([s,t])| + D_{\phi} \right)
= \frac{n-2}{n}D_{\phi}\footnote{Here we use the equality $\sup_{s,t \in N}|\phi([s,t])| = D_{\phi}$ \cite{Ba,Ca}.}
\end{equation}
for every $G$-invariant homogeneous quasimorphism $\phi:N \rightarrow \R$.
\end{corollary}

When $N=G$, \eqref{eqn:scl} is nothing but \cite[Theorem 2]{IMT1}. Since $scl_G(g)\leq scl_{G,N}(g)$, Corollary \ref{cor:mix-scl} gives stronger restriction.

\begin{remark}
It is known that $ \sup_{\phi} \frac{|\phi(g)|}{2D_{\phi}} =scl_{G,N}(g)$ where $\phi$ runs all $G$-invariant homogeneous quasimorphism which is not a homomorphism (Bavard's duality; \cite{Ba,KKMMM1}) so \eqref{eqn:q-hom} follows from \eqref{eqn:scl} and vice versa. We remark that our argument does not use these results whose proof uses Hahn-Banach theorem.
\end{remark}

In a similar vein, we have the following variant of Theorem \ref{theorem:G-inv-norm}.
 
\begin{proposition}
\label{prop:norm-variant}
Let $N$ be a normal subgroup of a group $G$.
Assume that $g \in N$ and $n \in t_G(g)$. Then for a symmetric $G$-invariant norm $\nu$ of $N$, 
 $\nu(g^{n}) \leq (n-1)\sup\{\nu([x,g]) \: | \: x \in G\}$
holds. In particular,
\[ \overline{\nu}(g) \leq \frac{n-1}{n}\sup\{\nu([x,g]) \: | \: x \in G\} \]
\end{proposition}
\begin{proof}
Assume that 
\[ g^{-1} = g^{x_1}\cdots g^{x_{n-1}} \]
for some $x_1,\ldots,x_{n-1} \in G$. 
Then 
\[ g^{-n} =[x_1,g]^{\ast} \cdots [x_{n-1},g]^{\ast} \]
where $[x_i,g]^{*}$ means a suitable conjugate of $[x_i,g]$.
\end{proof}

\begin{example}[(Stable) $\gamma_k$-length  \cite{CZ}]
The \emph{$\gamma_k$-length} $\ell_{\gamma_k}$ of a group $G$ is the minimum number the $k$-th commutator $[g_1,[g_2,[\ldots,[g_{k-1},g_k]]\cdots]]$ $(g_i \in G)$ and its inverses that is needed to express $g$.
The $\gamma_k$-length is a $G$-invariant norm on the $(k-1)$th lower central subgroup $\Gamma_{k-1}G=[G,[G,[\ldots,[G,G]]\cdots ]]$.

Assume that $g \in \Gamma_{k}G$ and $n \in t_G(g)$. If $\ell_{\gamma_{k-1}}(g)=1$, then $\ell_{\gamma_k}([x,g])=1$ for all $x \in G$ so Proposition \ref{prop:norm-variant} shows that the stable $\gamma_k$-length satisfies 
\[ \overline{\ell_{\gamma_k}}(g) \leq \frac{n-1}{n}. \]

\end{example}

\section{Alexander polynomial criterion}

\subsection{Alexander polynomial of modules}

We quickly review the Alexander polynomial. We refer to \cite{Hi} for algebraic treatments of Alexander polynomial.

Let $\Lambda= \Z[t_1^{\pm 1}, \ldots, t_s^{\pm 1}]$ be the Laurent polynomial ring of $s$ variables. 
For $f,g \in \Lambda$, we denote by $f \doteq g$  if $f = ug$ where $u \in \Lambda$ is a unit of $\Lambda$.
We will write an element $f \in \Lambda$ as $f(\bt) = \sum_{I}a_I \bt^{I}$, where $I=(i_1,\ldots, i_s)$ is a multi-index and $\bt^{I} = t_1^{i_1}t_2^{i_2} \cdots t_s^{i_s}$. We denote by $\varepsilon:\Lambda \rightarrow \Z$ the \emph{augmentation map} $\varepsilon( \sum_{I}a_I \bt^{I}) = \sum_I a_I$ (i.e. $\varepsilon(f(\bt))=f(1,\ldots,1)$).

A $\Lambda$-module $M$ is \emph{finitely presented} if there is an exact sequence of $\Lambda$-modules 
\[ \Lambda^{m} \stackrel{A}{\rightarrow}\Lambda^{n} \rightarrow M \rightarrow 0. \]
called a \emph{finite presentation} of $M$. The matrix $A$ is called a \emph{presentation matrix} of $M$. 

\begin{definition}[Elementary ideal and Alexander polynomial]
Let $M$ be a finitely presented $\Lambda$-module and $A$ be its presentation matrix.
The \emph{$k$-th elementary ideal} $E_k(M)$ is the ideal of $\Lambda$ generated by $(n-k)$ minors of $A$ (when $k>n$, we define $E_k=\{1\}$). The $k$-th \emph{Alexander polynomial} $\Delta_k(M) \in \Lambda$ is the generator of the smallest principal ideal of $R$ that contains $E_k(M)$.
\end{definition}

It is known that the elementary ideal does not depend on a choice of presentation matrix and that $\Delta_k(M)$ is uniquely determined up to multiplication of units of $\Lambda$.

Let 
\[ TM = \{m \in M \: | \: f m =0 \mbox{ for some } 0 \neq f \in \Lambda\}\] be the torsion submodule of $M$. The rank of $M$ is defined by
\[ \rank(M) = \dim_{k} k\otimes M\]
where $k$ is the quotient field of $\Lambda$.
We say that $M$ is a $\Lambda$-torsion module if $M=TM$, which is equivalent to $\rank(M)=0$.

The annihilator ideal of $m \in M$ is an ideal of $\Lambda$ defined by
\[ \Ann(m)=\{f \in \Lambda \: | \: fm=0\}. \]
Similarly, the annihilator ideal of $M$ is defined by
\[ \Ann(M)= \bigcap_{m \in M} \Ann(m) = \{f \in \Lambda \: | \: fm=0 \mbox{ for all } m \in M\}. \]

The Alexander polynomial and annihilator ideals are related as follows. 
\begin{proposition}
\label{prop:Alex-Ann}
{$ $}
\begin{itemize}
\item[(i)] $\Delta_{\rank(M)+k}(M) \doteq \Delta_k(TM)$ \cite[Theorem 3.4]{Hi}.
\item[(ii)] $\sqrt{\Ann(M)} = \sqrt{E_0(M)}$ \cite[Theorem 3.1]{Hi}.
\item[(iii)] If $M=TM$ and $M$ has a square presentation matrix, then $\Ann(M)=((\Delta_0(M)\slash \Delta_1(M))$ \cite[Corollary 3.4.1]{Hi}
\end{itemize}
Here $\sqrt{I}:=\{g \in \Lambda \: | \: g^{n} \in \Lambda \mbox{ for some }n>0\}$ is the radical of the ideal $I$.
\end{proposition}

We will use the following result later.

\begin{lemma}\cite[Theorem 3.12 (3)]{Hi}
\label{lemma:div}
Let $\Phi: M \rightarrow N$ be a homomorphism of $\Lambda$-modules.
If $\Phi|_{TM}:TM \rightarrow TN$ is a surjection, then $\Delta_{0}(TN)$ divides $\Delta_0(TM)$.
\end{lemma}

\subsection{Alexander tuples}

Let $N$ be a normal subgroup of a group $G$ such that its quotient group $G\slash N =\Z^{s}$ for some $s\geq 0$. The quotient group $G\slash N=\Z^{s}$ acts on the $1$st homology group $H_1(N;\Z)=N\slash [N,N]$ by conjugation. Hence $N\slash[N,N]$ has a structure of $\Lambda:=\Z[\Z^{s}]$-module. This $\Lambda$-module is called the \emph{Alexander module} and its Alexander polynomial is called the \emph{Alexander polynomial} of a group $G$. As we already mentioned, we regard them as an invariant of a pair $(G,N)$ of a group $G$ and its normal subgroup $N$ with $G\slash N = \Z^{s}$.

We slightly extend this construction.

\begin{definition}[Alexander tuple]
\label{def:Alex-tuple}
Let $G$ be a group and $X,N,H$ be normal subgroups of the group $G$.
We say that a tuple $\mathcal{A}=(G;(X,N,H))$ is an \emph{Alexander tuples} if they satisfy the following conditions.
\begin{itemize}
\item[(a)] $[H,N] \subset X \subset N \subset H$
\item[(b)] The quotient group $G\slash H$ is the free abelian group $\Z^{s}$ for $s \geq 0$.
\end{itemize}
\end{definition}

For an Alexander tuple $\mathcal{A}=(G;(X,N,H))$, we put 
\[ \Lambda=\Lambda_{\mathcal{A}} = \Z[G\slash H]= \Z[t_1^{\pm 1}, \ldots, t_s^{\pm 1}] \]
and
\[ M=M_{\mathcal{A}} = N\slash X \]
Since $[N,N] \subset [H,N] \subset X$, $M$ is an abelian group.
The group $G$ acts on $M$ by conjugation because $X$ and $N$ are normal. Furthermore, $[H,N] \subset X$ implies that the conjugation action of $H$ on $M$ is trivial. Thus the quotient group $G \slash H = \Z^s$ acts on $M$ hence $M$ is a $\Lambda$-module. 

\begin{definition}[Alexander module and polynomials of the Alexander tuple]
We call the $\Lambda$-module $M=M_{\mathcal{A}}$ \emph{the Alexander module} of the Alexander tuple $\mathcal{A}=(G;(X,N,H))$. 
We call the $0$-th Alexander polynomial $\Delta_{0}(TM_{\mathcal{A}})$ the \emph{Alexander polynomial} of the Alexander tuple $\mathcal{A}$ and denote by $\Delta_{\mathcal{A}}(\bt)$.
\end{definition}

The Alexander polynomial is usually used as an invariant of knots (and links) in the following manner.

\begin{example}[Alexander polynomial of a knot]
\label{example:A-knot}
Let $K$ be a knot in $S^{3}$ and $G=G(K)=\pi_1(S^{3} \setminus K)$ be the knot group, the fundamental group of its complement.
The $0$th Alexander polynomial of the Alexander tuple $(G;([[G,G],[G,G]],[G,G],[G,G]))$ is called the Alexander polynomial of a knot $K$ denoted by $\Delta_K(t)$. 
\end{example}

\subsection{Alexander polynomial and generalized torsion equation spectrum}

To state our theorem we introduce a notion of generalized torsion equation spectrum for an element of $\Lambda$-modules.

\begin{definition}
For a $\Lambda$-module $M$ and $m \in M$, the \emph{generalized torsion equation spectrum} of $m$ by
\[ t(m) = \{ \varepsilon(f(\bt))\: | \: f \in Ann(m) \mbox{ and } f \mbox{ is a positive element} \}\]

Here we say that an element $f(\bt)= \sum_{I}a_I \bt^{I} \in \Lambda$ is \emph{positive} if $f(\bt) \neq 0$ and $a_I \geq 0$ for all $I$. 
\end{definition}

For the Alexander module of Alexander tuples, the generalized torsion equation spectrum is nothing but the generalized torsion equation spectrum of suitable quotient group.

\begin{lemma}
\label{lemma:t-for-Amod}
Let $\mathcal{A}=(G;(X,N,H))$ be an Alexander tuple and $M_{\mathcal{A}}$ be its Alexander module. Then for $g \in N$, $t(p(g))=t_{G \slash X}(q(g))$ where $p:N \rightarrow M_{\mathcal{A}}=N\slash X$ and $q:G\rightarrow G\slash X$ are the quotient maps.
\end{lemma}
\begin{proof}
By definition, $n \in t_{G\slash X}(q(g))$ if and only if there exists $g_1,\ldots,g_n \in G$ such that 
\[ g^{g_1}\cdots g^{g_n} \in X\]
Since $g \in N$, by taking a projection map $p: N \rightarrow M_{\alpha}$ it is equivalent to
\[ p(g^{g_1}) \cdots p(g^{g_n}) = \left(\sum_{i=1}^{n}\phi(g_i)\right)p(g) = 0 \in M_{\alpha} = N \slash X,\]
where $\phi:G \rightarrow G\slash H \subset \Lambda$ is the projection map.
Therefore $n \in t_{G\slash X}(q(g))$ if and only if $n \in t(p(g))$.
\end{proof}

\begin{definition}
For an irreducible element $h=h(\bt) \in \Lambda$, we define 
\[ t(h)=t(h(\bt)) = \{ \varepsilon(f(\bt))\: | \: f(\bt) \in (h(\bt)), f \mbox{ is a positive element} \}\]
\end{definition}

The set $t(h)$ for the case $s=1$ will be studied in the next section. Now we are ready to prove the main theorem stated in the introduction.

\begin{theorem}
\label{theorem:main}
Let $\mathcal{A}=(G;(X,N,H))$ be an Alexander tuple. For an element $g \in N$, if $g \not \in X$, then there exists a irreducible factor $h(\bt)$ of $\Delta_\mathcal{A}(\bt)$ such that
\[ t(g) \subset t(h(\bt)) \]
\end{theorem}
\begin{proof}
If $t(g) = \emptyset$ we have nothing to prove so we assume that $t(g) \neq \emptyset$. We put $m=p(g)$ where $p: N \rightarrow M=N \slash X$ is the quotient map. 
By Lemma, $t(g) \subset t(m)$ hence $t(m) \neq \emptyset$. In particular, $m \in TM$.

Let $\Lambda m$ be the sub $\Lambda$-module of $TM$ generated by $m$. By Proposition \ref{prop:Alex-Ann},
\[
\begin{matrix}
\Ann(m) & = & \Ann(\Lambda m) & \subset & \sqrt{\Ann(\Lambda m)} & = & \sqrt{E_0(\Lambda m)} & \subset& \sqrt{(\Delta_{0}(\Lambda m))} \\
& & \cup & & \cup & & \cup & & \cup\\
& & \Ann(TM) & \subset & \sqrt{\Ann(TM)} & = & \sqrt{E_0(TM)} & \subset &  \sqrt{(\Delta_{0}(TM))} \\
& & & & & & \cup & &\cup \\
& & & & & & E_0(TM) & \subset & (\Delta_0(TM))
\end{matrix}
\]
Thus $\sqrt{(\Delta_{0}(\Lambda m))}$ is a principal ideal that contains $(\Delta_0(TM))$. Since we are assuming $g \not \in X$, $m=p(g)\neq 0$. Thus  $\sqrt{(\Delta_{0}(\Lambda m))}$ is not the whole $\Lambda$. 
Therefore there exists a non-trivial irreducible factor $h(\bt)$ of $\Delta_0(TM)$ such that 
\[ \Ann(m) \subset \sqrt{(\Delta_{0}(\Lambda m))} \subset (h(\bt)) \]
hence
\[ t(g) \subset t(m) \subset t(h(\bt)) \]
\end{proof}

Although we are using the Alexander polynomial $\Delta_0(TM)$, if we know the structure of the annihilator ideals we can often improve the theorem as the next example shows.

\begin{remark}
If $M=TM$ and $M$ has a square presentation matrix
\[ t(g) \subset t(h(\bt)) \]
for some irreducible factor $h(\bt)$ of $\lambda_0(M)=\Delta_0(M_{\mathcal A}) \slash \Delta_1(M_{\mathcal A})$, because $\Ann(M) =(\lambda_0(M))$ by Proposition \ref{prop:Alex-Ann} (iii).
\end{remark}

For normal subgroups $N$ and $H$ of $G$ such that $N \subset H$, the \emph{$H$-lower central series} of a normal subgroup $N$ 
\[ \gamma^{H}_0 N \supset \gamma^{H}_1 N \supset \cdots \supset \gamma^{H}_k N \supset \gamma^{H}_{k+1} N \supset \cdots \]
by 
$\gamma^{H}_0 N = N$ and $\gamma^{H}_{k+1} N = [H,\gamma^{H}_{k}N]$. When $N=H$, this is the usual lower central series of $H$. We put $\gamma^{H}_{\infty} N = \bigcap_{k\geq 0} \gamma^{H}_k N$. Then iterated use of Theorem \ref{theorem:main} for the Alexander tuple $(G;(\gamma^{H}_{k+1}N,\gamma^{H}_k N, H))$ gives the following.

\begin{corollary}
\label{cor:Alex-H-series}
Let $N$ and $H$ be normal subgroups of a group $G$ that satisfy the conditions
\begin{itemize}
\item[(a)] $N \subset H$.
\item[(b)] The quotient group $G\slash H$ is the free abelian group $\Z^{s}$ for $s \geq 0$.
\end{itemize}
For $k>0$, let $\mathcal{A}_k= (G;(\gamma^{H}_{k+1}N,\gamma^{H}_kN, H))$ be an Alexander tuple.
If $g \in \Gamma^{H}_k N$ and $g \not \in \Gamma^{H}_{k+1} N$, then there exists an ireducible factor $h(\bt)$ of $\Delta_{\mathcal{A}_k}(\bt)$ such that
\[ t(g) \subset t(h(\bt)) \]
\end{corollary}

We give the simplest application, the case $H=G$. For a prime number $p$, a group $G$ is \emph{residually finite $p$} if for every non-trivial $g \in G$, there exists a surjection $f:G \rightarrow Q$ to a finite $p$-group $Q$ such that $f(g) \neq 1$.

\begin{corollary}
If $G$ is a residually finite $p$-group, then for every non-trivial $g \in G$, $t(g) \subset p\N$.  
\end{corollary}
\begin{proof}
By the monotonicity (Lemma \ref{lemma:monotonicity}) it is sufficient to show the assertion for finite $p$-groups. 
Let $G$ be a finite $p$-group. We apply Corollary \ref{cor:Alex-H-series} for $N=H=G$ (thus $\Lambda=\Z$). Since a finite $p$-group is nilpotent, there exists $k\geq 0$ such that $ g \in \gamma_k G$ but $g \not \in \gamma_{k+1}G$. 
Since $G$ is a finite $p$-group, $\gamma_k G \slash \gamma_{k+1}G$ is an abelian $p$-group hence $t(g) \subset p\mathbb{N}$.
\end{proof}

Finally, we give useful variant of Corollary \ref{cor:Alex-H-series} that only uses one Alexander module $M_{\mathcal{A}}$ for $\mathcal{A}=(G;([N,H],N,H))$.

For a $\Lambda$-module $M$, let $M^{\otimes k}$ be the tensor product of the $\Z$-module (i.e. abelian group) $M$. We view $M^{\otimes k}$ as a $\Lambda$-module by the diagonal action; for $\bt \in \Z^s$ and $m_1,\ldots,m_k \in M$, we define
\[ \bt(m_1\otimes m_2 \otimes \cdots \otimes m_k) = (\bt m_1\otimes \bt m_2 \otimes \cdots \otimes \bt m_k).\]

\begin{corollary}
\label{cor:Alex-var3}
Let $\mathcal{A}=(G;([N,H],N,H))$ be an Alexander tuple and $M=N \slash [N,H]$ be its Alexander module. If $M$ is a $\Lambda$-torsion module, then for every $g \in G$, if $g \not \in \gamma^{H}_{\infty} N$ then $t(g) \subset t(h(\bt))$ for some irreducible factor $h(\bt)$ of $\Delta_{0}(M^{\otimes k})$. 
\end{corollary}
\begin{proof}
Let $M_k$ be the Alexander module of the Alexander tuple $\mathcal{A}_k= (G;(\gamma^{H}_{k+1}N,\gamma^{H}_kN, H))$ and $f:M^{\otimes (k+1)} \rightarrow M_k$ be the map defined by 
\[ f(a_1 \otimes a_2 \otimes \cdots \otimes a_{k+1}) = [a_1,[a_2,[\cdots,[a_{k},a_{k+1}] \cdots]]] \qquad (a_i \in N)\]
The map $f$ is a surjective $\Lambda$-module homomorphism (see \cite[5.2.5]{Ro}).
Since we are assuming $M$ is a $\Lambda$-torsion module, so is $M^{\otimes k}$. Therefore by Lemma \ref{lemma:div}, $\Delta_0(TM_k)$ divides $\Delta_0(TM^{\otimes k}) = \Delta_{\mathcal{A}_k}(\bt)$ so the assertion follows from Corollary \ref{cor:Alex-H-series}.
\end{proof}

\begin{remark}[Non-commutative settings]
Throughout this section we assume that $G\slash H$ is a finitely generated free abelian group (assumption (b) of the Alexander tuple). However the arguments presented in this section works without this assumption. 

Let $X,N,H$ be normal subgroups of $G$ such that $[N,H] \subset X \subset N \subset H$. Let $\Lambda:=\Z[G\slash H]$ be the group ring of the (possibly non-commutative) quotient group $G\slash H$. Then the conjugation of $G$ induces a structure of right $\Lambda$-module for the quotient group $M=N \slash X$. 

By the same argument, we see that if $g \in N$, then 
\[ t(g) \subset  t_{G\slash X}(q(g)) = \{\varepsilon(f) \: | \:  f \in \Ann(p(g)) \mbox{ and } f \mbox{ is a positive element} \} \]
where $p:N \rightarrow N \slash X$ and $q: G \rightarrow G \slash X$ are the projection maps.

Unfortunately, it is not easy to use this non-commutative version.
\end{remark}

\section{The set \texorpdfstring{$t(h)$}{t(h)}}

To utilize the results in the previous section, we need to know set $t(h(\bt))$ for irreducible $h(\bt)$. In this section we discuss the structure of the set $t(h(t))$ for irreducible one-variable Laurent polynomial $h \in \Lambda=\Z[t^{\pm 1}]$.

To begin with, we observe the following simple properties.
For a positive integer $k>0$, let $\Phi_k(t)$ be the $k$-th cyclotomic polynomial and $P_k$ be the set of roots of $\Phi_k$, the set of primitive $k$-th root of unities.

\begin{lemma}
\label{lemma:elementary-estimate}
Assume that $h(t)=a_mt^{m}+ \cdots +a_1t+a_0$ ($a_0,a_m \neq 0$, $m\geq 1$) is irreducible.
\begin{itemize}
\item[(i)] $t(h)\neq \emptyset$ if and only if $h$ has no positive real root.
\item[(ii)] $t(h) \subset |h(1)|\Z$.
\item[(iii)] If $n \in t(h(t))$ then $n \geq |a_m|+|a_0|$.
\item[(iv)] $2 \in t(h(t))$ if and only if $h(t)=\Phi_{2s}$ for some $s>0$. 
\end{itemize}
\end{lemma}
\begin{proof}
(i) This is proven in \cite{Du} (see also \cite{Br}).

(ii) If $n \in t(h(t))$ then there exists $g(t)=b_kt^{k} + \cdots + b_0$ ($b_k,b_0 \neq 0$) such that $g(t)h(t)$ is positive and that $n=f(1)$. Then $n=g(1)h(1)=|g(1)||h(1)|$ so $|h(1)|$ always divides $n$.

(iii) Since $f(t)=g(t)h(t)$ is positive,
\[ n=a_mb_k + \cdots +a_0b_0 \geq |a_mb_k| + |a_0b_0| \geq |a_m|+|a_0| \]
(iv) If $2 \in O(h(t))$ there exists $g(t) \in \Lambda$ such that $g(t)h(t) = 1 + t^{d}$ or $g(t)h(t)=2$. Teh latter case does not happen since we are assuming that $h$ is not a constant. Thus $h(t)$ divides $1+t^{d}$ which implies that $h(t)=\Phi_{2s}$ for some divisor $s$ of $d$.
\end{proof}

To get more constraints, we use the following quantity.

\begin{definition}
For $h(t) \in \Lambda$, we define
\[ R_k(h) =  \prod_{\zeta \in P_k}|h(\zeta)| \in \Z_{\geq 0} \]
where $P_k$ is the set of primitive $k$-th root of unities.
\end{definition}
This is the absolute value of the resultant of $h(t)$ and the $k$-th cyclotomic polynomial $\Phi_k(t)$. 

\begin{proposition}
\label{prop:AT-root-unity}
Let $h \in \Lambda$ be an irreducible polynomial and $k=p^e$ be a power of a prime $p$.
Then if $n \in t(h)$, then either
\begin{itemize}
\item[(a)] $p$ divides $n$. Furthermore, if $h \neq \Phi_k(t)$ then $p|h(1)|$ divides $n$, or, 
\item[(b)] $n^{\phi(k)} \geq R_k(h)$. 
\end{itemize}
holds. Here $\phi(k):=\# P_k$ is the euler totient function.
\end{proposition}
\begin{proof}
Assume that $f(t)=g(t)h(t)$ is positive and that $n= f(1)=g(1)h(1)$.
Since $f(t)$ is positive, $n=|f(1)| \geq |f(\omega)|$ for all $\omega \in \{z \in \mathbb{C} \: | \: |z| = 1\}$.
In particular, $n=f(1) \geq |f(\zeta)|$ for every root of unities $\zeta$. Therefore
\[ n^{\phi(k)}=|f(1)|^{\phi(k)} \geq R_k(f) = R_k(g)R_k(h) \]
holds for all $k>0$. 

If $R_k(g)=0$, then we may write $f(t)=\Phi_k(t)f^{*}(t)$ for some $f^{*}(t)$.
Since $\Phi_{k}(1)=p$ if $k=p^{e}$, 
\[ n= f(1) = |\Phi_k(1)||f^{*}(1)| = p |f^{*}(1)| \]
Furthermore, if $h \neq \Phi_k(t)$ then $f(t)=\Phi_k(t)g^{*}(t)h(t)$ for some $g^{*}(t)$ hence
\[ n = f(1) = |\Phi_k(1)||g^{*}(1)||h(1)| = p |h(1)||g^{*}(1)|. \]
Thus, in this case (a) holds. 
If $R_k(g) \neq 0$ then $R_k(f)=R_k(g)R_k(h) \geq R_k(h)$ so (b) holds.  
\end{proof}

The \emph{Mahler measure} $M(f)$ of a polynomial $f(t)=a_dt^{d}+a_{d-1}t^{d-1}+\cdots+a_0 \in \Z[t^{\pm 1}]$ is defined by
\[ M(f) = |a_d|\prod_{i=1}^{d}\max\{1,|\alpha_i|\} \]
where $\alpha_1,\ldots,\alpha_d$ are zeros of $f(t)$.
It is known that 
\cite{GS,Ri,SW}
\[ \lim_{k \to \infty} \left(\prod_{\zeta^{k}=1} h(\zeta)\right)^{\frac{1}{k}} = \left( \prod_{d | k} R_d(h) \right)^{\frac{1}{k}} = M(h) \]
Thus by Proposition \ref{prop:AT-root-unity} we get the following.

\begin{corollary}[Mahler measure bound]
If $h(t)$ is irreducible, then $n \geq M(h)$ for all $n \in t(h)$.
\end{corollary}

We give some simple calculations which will be used later.

\begin{example}
\label{example:cyclotomic}
Let $k=p^{a}$ be a power of a prime $p$. Since $\Phi_{k}(1)=p$, by Lemma \ref{lemma:elementary-estimate} $t(\Phi_{k}) \subset p\N$. Indeed, $\Phi_k = t^{(p-1)p^{a-1}}+ t^{(p-2)p^{(a-1)}} + \cdots +1$ is positive so $p \in t(\Phi_k)$. Thus 
\[ t(\Phi_{p^{a}})=p\N \]

Let $k=p^{a}q^{b}$ where $p<q$ be primes and $a,b>0$.
Assume that $n \in t(\Phi_k)$.
Since $R_{p}(\Phi_{k})=q$ (see \cite{Ap}), by Proposition \ref{prop:AT-root-unity} $p$ divides $n$, or, $n \geq q$. 
Similarly, since $R_{q}(\Phi_{k})=p$, either $q$ divides $n$, or, $n \geq p$.
Since $p<q$, we conclude that 
\[ t(\Phi_{p^aq^b}) \subset p\N \cup \N_{\geq q} \]
\end{example}

\section{Application: generalized torsion elements of knot groups}

In this section we apply our arguments for the case of knots groups. We refer to \cite{CR} as a reference for the knot theory and its relation to orderable group theory.

Let $K$ be a knot in $S^{3}$, and $G=G(K):=\pi_1(S^{3} \setminus K)$ be the knot group, the fundamental group of the knot complement. As we have mentioned in Example \ref{example:A-knot}, the Alexander polynomial $\Delta_K(t)$ of the knot $K$ in knot theory is the Alexander polynomial of Alexander tuple $(G;([[G,G],[G,G]], [G,G],[G,G]))$.

Let $\Sigma_{k}(K)$ be the $k$-fold cyclic branched covering of $K$.
If $\Sigma_k(K)$ is a rational homology sphere (that is equivalent to saying that $\Delta(\zeta) \neq 0$ for every (not necessarily primitive) $k$-th root of unities), then the order of homology is given by
\[|H_1(\Sigma_k(K);\Z)| = \prod_{i=1}^{k}|\Delta_K(\zeta^{i})| = \prod_{d | k} R_d(h)  \]
\cite{We}. Thus Proposition \ref{prop:AT-root-unity} leads to the following interesting relation between $t(g)$ of a knot group $g \in G(K)$ and the (growth of) homology of cyclic branched covers. Here we state the case $\Delta_K(t)$ is irreducible, to make the statement simpler.

\begin{corollary}
\label{cor:tg-knot}
Let $K$ be a knot in $S^{3}$. Assume that $\Delta_K(t)$ is irreducible.
Let $G=G(K)=\pi_1(S^{3}\setminus K)$ be the knot group and $\Sigma_{k}(K)$ be the $k$-fold cyclic branched covering of $K$.
Assume that $\Sigma_k(K)$ is a rational homology sphere and $k=p^{e}$ is a power of a prime $p$. For $g \not \in [[G,G],[G,G]]$ and $n \in t(g)$ either
\begin{itemize}
\item[(a)] $n \geq |H_1(\Sigma_k(K);\Z)|^{\frac{1}{k-1}}$, or,
\item[(b)] $p$ divides $n$.
\end{itemize}
holds.
\end{corollary}

The following special case ($k=2$) of Corollary \ref{cor:tg-knot} deserves to mention.

\begin{corollary}[Determinant bound]
\label{cor:determinant}
Let $G=G(K)$ be the knot group of a knot $K$. Assume that $\Delta_K(t)$ is irreducible.
If $g \not \in [[G,G],[G,G]]$ and $n \in t(G(K))$ is odd, then
\[ n \geq \det(K)=|\Delta_K(-1)| \] 
\end{corollary}

It is known that for a given $k$, the number of alternating knot $K$ satisfying $\det(K)\leq k$ is finite \cite{Ban}. Thus Corollary \ref{cor:determinant} says that for each odd $k$, there are only finitely many alternating knots $K$ having a generalized torsion element $g \not \in [[G,G],[G,G]]$ with $gord(g)=k$ and $\Delta_K(t)$ is irreducible. This partially explains why finding a generalized torsion element of a knot group is difficult, even though it is expected that many knots admit a generalized torsion element.

This observation poses the following finiteness question.
\begin{question}
\label{ques:finite}
For a given integer $m$, let $M_{\sf alt}(m)$ be the number of alternating knot $K$ other than $(2,k)$-torus knot\footnote{This is equivalent to saying that $K$ is a hyperbolic alternating knot \cite{Me}} whose knot group $G(K)$ has a generalized torsion element $g$ with $gord(g)=m$. Similarly, let $M_{\sf hyp}(m)$ be the number of hyperbolic knot $K$ whose knot group $G(K)$ has a generalized torsion element $g$ with $gord(g)=m$.
 Is $M_{\sf alt}(m)$ (resp. $M_{\sf hyp}(m)$) finite ?
\end{question}

A similar question makes sense for other appropriate classes of knots, although it is necessary to exclude torus knots and cable knots as the next proposition shows.

\begin{proposition}
\label{prop:torus-cable}
let $K$ be a $(p^{a},q^{b})$-torus knot $T_{p^{a},q^{b}}$ or a $(p^{a},q^{b})$-cable knot where $p<q$ are primes. Then its knot group $G(K)$ admits a generalized torsion elements with $gord(g)=p$ such that $g \not \in [[G(K),G(K)],[G(K),G(K)]]$.
\end{proposition}
\begin{proof}
Assume that $K$ is a $(p^{a},q^{b})$-cable of a knot $C$ where we allow $C$ to be the trivial knot (in such case, $K$ is just the $(p^{a},q^{b})$-torus knot).
Then its knot group is the amalgamated free product $G(K) = G(C) \ast_{\mu^{q^b}\lambda^{p^a}=y^{p^a}} \langle y \rangle$, where $\mu$ and $\lambda$ are the meridian and the longitude of $P$.
Since $[\mu,y^{p^{a-1}}] \neq 1$ but $[\mu,y^{p^{a}}]=[\mu,\mu^{q^b}\lambda^{p^a}]=1$, by Lemma \ref{lemma:generalized-torsion-simple} $p \in t_{G(K)}([\mu,y^{p^{a-1}}])$. Furthermore, $[\mu,y^{p^{a-1}}] \not \in \not \in [[G(K),G(K)],[G(K),G(K)]]$.

Let $\pi:G(C)\rightarrow \Z$ be the projection map, $x = \pi(\mu)$ and let $G=G(T_{a,b})$ be the torus knot group. The projection induces the epimorphism 
\[ \pi:G(K)=G(C) \ast_{\mu^{q^b}\lambda^{p^a}=y^{p^a}}\Z \rightarrow \Z \ast_{x^{q^b} = y^{p^a}} \Z = G(T_{p^a,q^b})=G\]
Since $\Delta_{T_{p^{a},q^{b}}}(t)=\frac{(t^{p^aq^{b}}-1)(t-1)}{(t^{p^{a}}-1)(t^{q^{b}}-1)}$, the irreducible factors of $\Delta_{T_{p^{a},q^{b}}}(t)$ are cyclotomic polynomial $\Phi_{p^{a'}q^{b'}}(t)$ for some $0\leq a' \leq a, 0 \leq b' \leq b$. Since $f(\pi([\mu,y^{p^{a-1}}])) \not \in [[G,G],[G,G]]$, by Theorem \ref{theorem:main}, Example \ref{example:cyclotomic} and the monotonicity 
\[ t_{G(K)}([\mu,y^{p^{a-1}}]) \subset t_{G}(\pi([\mu,y^{p^{a-1}}])) \subset p\N \cup \N_{\geq q}.\]
Therefore $gord_{G(K)}([\mu,y^{p^{a-1}}])=p$.
\end{proof}

To get restrictions of $t(g)$ for an element $g \in [[G,G],[G,G]]$ from the Alexander polynomials we use Corollary \ref{cor:Alex-H-series} or Corollary \ref{cor:Alex-var3}. Here we give one situation where we can utilize Corollary \ref{cor:Alex-var3} effectively.

\begin{theorem}
\label{theorem:knot}
Let $K$ be a knot in $S^{3}$ and $G=G(K)$ be its knot group.
Assume that $[G,G]$ is residually torsion-free nilpotent and that $\deg \Delta_K(t)=2g(K)$ where $g(K)$ is the genus of $K$.
If $\Delta_K(t)$ divides $(t^{k}-1)$ where $k=p^{a}q^{b}$ for some distinct primes $p,q$ ($p<q$), then for every $g \in G$, 
\[ t(g) \subset p\N \cup \N_{\geq q}\]
\end{theorem}
\begin{proof}
Since $\Delta_K(t)$ divides $(t^{k}-1)$, $\Delta_K(t)$ is monic. A knot having the properties that $\deg \Delta_K(t)=2g(K)$  and that $\Delta_K(t)$ is monic is called (integrally) \emph{homologically fibered knot}. For such a knot, the Alexander module $M$ of $G(K)$ has a square presentation matrix of the form $A=tI_{2g}-S$ where $S$ is certain $2g \times 2g$ integer matrix \cite{GoSa}.

By the definition of the tensor product module $M^{\otimes m}$, $M^{\otimes m}$ has a presentation matrix $A_m=tI_{(2g)^{m}}-S^{\otimes m}$ where $S^{\otimes m}: (\Z^{2g})^{\otimes m} \rightarrow (\Z^{2g})^{\otimes m}$ is the tensor product of $S^{\otimes m}$.

Let $\alpha_1,\ldots,\alpha_{2g} \in \C$ be the roots of the Alexander polynomial $\Delta_K(t)$. Then for $m \geq 1$
\[ \Delta_0(M^{\otimes m}) =\det(tI_{(2g)^{m}}-S^{\otimes m}) = \prod_{i_1=1}^{2g}\prod_{i_2=1}^{2g}\cdots \prod_{i_m=1}^{2g} (t-\alpha_{i_1}\alpha_{i_2}\cdots \alpha_{i_m})\]

Since $\Delta_K(t)$ divides $t^{k}-1$, $\alpha_{1},\ldots,\alpha_{2g}$ are (not necessarily primitive) $k$-th root of unities. Thus their product $\alpha_{i_1}\alpha_{i_2}\cdots \alpha_{i_m}$ are also the $k$-th root of unities. 
Since $k=p^{a}q^{b}$, every irreducible factor of $\Delta_0(M^{\otimes m})$ is a cyclotomic polynomial $\Phi_{p^{a'}q^{b'}}$ ($0\leq a' \leq a$, $0 \leq b' \leq b$).
Thus by Corollary \ref{cor:Alex-var3}
\[ t(g) \subset t(\Phi_{p^{a'}q^{b'}}) \]
for some $a',b'$ hence by Example \ref{example:cyclotomic}, $t(g) \subset p\N \cup \N_{\geq q}$.
\end{proof}

The most fundamental example of knots satisfying the assumptions $[G,G]$ is residually torsion-free nilpotent and that $\deg \Delta_K(t)=2g(K)$ of Theorem \ref{theorem:knot} is \emph{fibered knots}, a knot whose complements has a structure of a surface bundle over the circle.

Since the torus knot satisfies the assumptions, we get the following.
\begin{corollary}
\label{cor:torus-knot}
Let $K$ be the $(p^{a},q^{b})$-torus knot where $p<q$ are primes. Then for every $g \in G(K)$, $t(g) \subset p\mathbb{N} \cup \mathbb{N}_{\geq q}$. 
\end{corollary}

\begin{example}
Let $K$ be the $(2,5)$ torus knot $K=T_{2,5}$ or the knot $10_{132}$. They are fibered knots with the Alexander polynomial $\Delta_{T_{2,5}}(t)=t^{4}-t^{3}+t^{2}-t+1 = \Phi_{10}(t)$. By Theorem \ref{theorem:knot}, $t(g) \subset 2\mathbb{N} \cup \mathbb{N}_{\geq 5}$ for every $g \in G(K)$. In particular, $G(K)$ has no generalized torsion element of generalized torsion order $3$. 

On the other hand, the knot $T_{2,5}$ has a generalized torsion element of generalized torsion order $2$, whereas the knot $10_{132}$ has no generalized torsion element of generalized torsion order $2$ (because it is hyperbolic). 
\end{example}

Our result also can be used determine the generalized torsion order.
\begin{example}[Himeno's generalized torsion element \cite{Him}]
Let $K=T_{2,k}$ be the $(2,k)$-torus knot with $k>3$. Himeno showed for $n>0$, the element $E_n \in G(K)= \langle a,b \: | \: a^{2}=b^{k}\rangle = G(K)$ given by 
\[ E_n=[a,b]^n[a,b^{2}][a,b]^{n+1}[a,b^{k-1}],\]
satisfies $4 \in t(E_n)$. He verified $gord(E_1)=4$ for small $k$ by using computer calculations of scl \cite{Him}.
 
It is easy to check that $2 \not \in t(E_n)$ because $E_n$ and $E_n^{-1}$ are not conjugate. If $k=q^{e}$ is a power of a prime $q>3$, $3 \not \in t(E_n)$ by Theorem \ref{theorem:knot}. Therefore $gord(E_n)=4$ as expected.
\end{example}

A knot satisfying the condition $\deg \Delta_K(t)=2g(K)$ is called \emph{rationally homologically fibered knot}. In \cite{It} we showed that for a rationally homologically fibered knot $K$, if $\Delta_K(t)$ has no positive real root then $G=G(K)$ is not bi-orderable. 

On the other hand, our argument shows the following (though we state it for the knot group, the same is true for general groups).

\begin{proposition}
Let $K$ be a knot in $S^{3}$ and $G=G(K)$ be the knot group. If $\Delta_K(t)$ has no positive real root then $G \slash[[G,G],[G,G]]$ has a generalized torsion element.
\end{proposition}
\begin{proof}
This is an immediate consequence of Lemma \ref{lemma:t-for-Amod} and Lemma \ref{lemma:elementary-estimate} (i).
\end{proof}
This gives a partial answer to \cite[Question 6.6]{NR} where it asks the existence of generalized torsion if $G(K)$ when $\Delta_K(t)$ has no positive real root.

It is conjectured that for 3-manifold groups (that includes the knot groups), the bi-orderability is equivalent to the non-existence of generalized torsion element \cite{MT,IMT2}. In this prospect, it is interesting to ask
\begin{question}
\label{ques:Q-h-fiber}
If $K$ is rationally homologically fibered and $\Delta_K(t)$ has no positive real root, is there a generalized torsion element $g \in G=G(K)$ such that $g \not \in [[G,G],[G,G]]$ ?
\end{question}

\section{Generalized torsion order spectrum}

For a group $G$, the \emph{torsion order spectrum} $ord(G)$ is the set defined by
\[ ord(G) =\{ ord(g)\: | \: g \mbox{ is a torsion element of } G\}. \]
It is easy to see that $ord(G)$ cannot be arbitrary. because $ord(G)$ is \emph{factor-complete}, which means that if $pq \in t(G)$ with $p,q \neq 1$ then $p,q \in t(G)$. 

The torsion order spectrum is characterized as follows \cite{Ch}.

\begin{theorem}
\label{theorem:order-spectrum}
{$ $}
\begin{itemize}
\item[(i)] For a factor-complete subset $A$ of $\N_{\geq 2}$, there exists a finitely generated group $G$ such that $ord(G)=A$.
\item[(ii)] For a factor-complete subset $A$ of $\N_{\geq 2}$, there exists a finitely presented group $G$ such that $ord(G)=A$ if and only if $A$ is a  $\Sigma_2^{0}$-set (\cite[Theorem 6.3]{Ch}).
\end{itemize}
\end{theorem}
Here $\Sigma_2^{0}$-set is a set appeared in a theory of arithmetical hierarchy, and is larger than $\Sigma_1^{0}$-set, the recursively enumerable sets.

As a natural generalization of torsion order spectrum, it is natural to investigate the following set.

\begin{definition}
Let $G$ be a group.
The \emph{generalized torsion order spectrum} $go(G)$ of $G$ is  
\[ gord(G)= \{gord(g) \: | \: g \mbox{ is a generalized torsion element of } G\}. \]
The \emph{strict generalized torsion order spectrum} $gt^{*}(G)$ of $G$ is 
\[ gord^{*}(G)= \{gord(g) \: | \: g \mbox{ is a non-torsion, generalized torsion element of }G\}. \]
\end{definition}

Unlike torsion order spectrum, generalized torsion order spectrum is not necessarily  factor-complete as the next lemma shows.

\begin{lemma}
For every $n \in \Z_{\geq 2}$, there exists a finitely presented torsion-free group $G_n$ such that $gord^{*}(G_n)=gord(G_n)=\{n\}$.
\end{lemma}
\begin{proof}
When $n=2$, let $G_2 = \langle a,t \: | \: tat^{-1} = a^{-1}\rangle$ be the infinite diherdal group.
If $g \in G_2$ is a generalized torsion element, then $g \in \langle a \rangle$. However, every element in $\langle a \rangle$ is a generalized torsion element of generalized torsion order two because $ta^{k}t^{-1} \cdot a^{k} = 1$ for every $k$. Thus $gord^{*}(G_2)=\{2\}$.\\

For $n \geq 3$ and $n \neq 4$, let $A$ be the free abelian group of rank two generated by $a,b$ and let 
\[ G_n = \langle t, a,b \: | \: tat^{-1}=a^{-n+2}b, tbt^{-1}=a^{-1}, [a,b]= 1 \rangle. \]
$G_N$ is the HNN extension $1 \rightarrow A \rightarrow G \rightarrow \Z =\langle t \rangle$.
If $g \in G_n$ is a generalized torsion element, then $g \in \langle a ,b\rangle$. The Alexander polynomial of $G$ is an irreducible polynomial $t^{2}+(n-2)t +1$.
If $g \in G_n$ is a generalized torsion element, then $g \in [G_n,G_n]=A$. 
For every $1\neq g \in A$, $t^2gt^{-2}(tgt^{-1})^{n-2}g = 1$. Furthermore, by Theorem \ref{theorem:main} and Lemma \ref{lemma:elementary-estimate}, $t(g) \subset n\mathbb{N}$. Thus $gord^{*}(G_n)=\{n\}$.

The group $G_4$ is constructed in a similar manner; 
let $A$ be the free abelian group generated by $a,b,c$. We define 
\[ G_4 = \left\langle t, a,b,c \: \middle| \: \begin{matrix}
 tat^{-1}=b, tbt^{-1}=c, tct^{-1}=a^{-1}b^{-2}\\
 [a,b]=[a,c]=[b,c]=1 \end{matrix}
 \right\rangle. \]
$G_4$ is the HNN extension $1 \rightarrow A \rightarrow G \rightarrow \Z =\langle t \rangle \rightarrow 1$. Its Alexander polynomial is $t^{3}+2t+1$. Since it is irreducible, we conclude $gord^{*}(G_4)=\{4\}$ by the same argument.
\end{proof}

The (strict) generalized torsion order spectrum behaves nicely with respect to the free product.

\begin{theorem}
If $G$ and $H$ are torsion-free, then $gord^{*}(G\ast H) = gord^{*}(G) \cup gord^{*}(H)$.
\end{theorem}
\begin{proof}
By \cite[Theorem 1.5]{IMT1}, a generalized torsion element $x$ of $G\ast H$ is conjugate to a generalized torsion element of $G$ or $H$. Since generalized torsion order is invariant under conjugation, we assume that $x \in G \subset G\ast H$ (or $x \in H$). Since the inclusion map $G \hookrightarrow G\ast H$ is a retract, $gord_G(x)=gord_{G\ast H}(x)$ by Corollary \ref{cor:monotonicity}. Thus $gord^{*}(G\ast H) = gord^{*}(G) \cup gord^{*}(H)$.
\end{proof}

These two results gives the following realization result.

\begin{corollary}
\label{cor:realization}
For every subset $A \subset \N_{\geq 2}$, there exists a countable, torsion-free group $G$ such that $gord^{*}(G) = gord(G)=A$.
\end{corollary}

It is interesting to ask when we can take such a group $G$ finitely generated (or, finitely presented, with suitable complexity assumption on $A$). 
For a torsion spectrum case, Higman-Neumann-Neumann embedding theorem allows us to embed countable groups to finitely generated groups so that the set of torsion elements are the same. For generalized torsion case, we do not know similar embedding is possible or not.

\begin{question}[Higman-Neumann-Neumann embedding preserving generalized torsion equation spectrum/generalized torsion orders]
Let $G$ be a countable group. Is there an embedding of $G$ into a finitely generated group $H$ such that $t_G(g)=t_H(g)$ (or, $gord_G(g)=gord_H(g)$) for all $g \in G$? 
\end{question}

\end{document}